\documentclass[12pt,reqno]{amsart}
\usepackage{amsfonts}
\usepackage{amssymb, mathrsfs,  amsmath}
\def\sqr#1#2{{\vcenter{\vbox{\hrule height.#2pt
        \hbox{\vrule width.#2pt height#1pt \kern#2pt
        \vrule width.#2pt}
        \hrule height.#2pt}}}}

\DeclareMathOperator*{\argmin}{arg\,min}

\newcommand{\nc}{\newcommand}
\nc{\parent}[1]{$[\![#1]\!]$}
\newtheorem{theorem}{Theorem}[section]

\newtheorem{example}{Example}[section]
\newtheorem{corollary}{Corollary}[section]
\newtheorem{proposition}{Proposition}[section]
\newtheorem{remark}{Remark}

\nc{\sD}{{\mathscr D}} \nc{\sE}{{\mathscr E}}
\nc{\cadlag}{c\`{a}dl\`{a}g } \nc{\ba}{\begin{array}}
\nc{\ea}{\end{array}} \nc{\be}{\begin{equation}}
\nc{\ee}{\end{equation}} \nc{\bea}{\begin{eqnarray}}
\nc{\eea}{\end{eqnarray}} \nc{\bean}{\begin{eqnarray*}}
\nc{\eean}{\end{eqnarray*}} \nc{\bu}{\bullet} \nc{\nn}{\nonumber}
\nc{\cA}{{\mathcal A}} \nc{\cB}{{\mathcal B}} \nc{\cC}{{\mathcal C}} \nc{\bfE}{\mathbf{E}}
\nc{\cD}{{\mathcal D}} \nc{\bbD}{\mathbb{D}}\nc{\bbH}{\mathbb{H}}
\nc{\bbF}{\mathbb{F}}\nc{\bbG}{\mathbb{G}}\nc{\cG}{{\mathcal G}} \nc{\cF}{{\mathcal F}}
\nc{\cS}{{\mathcal S}} \nc{\cU}{{\mathcal U}} \nc{\cH}{{\mathcal H}}\nc{\cJ}{{\mathcal J}}
\nc{\cK}{{\mathcal K}} \nc{\cL}{{\mathcal L}} \nc{\cM}{{\mathcal M}}
\nc{\cO}{{\mathcal O}} \nc{\cP}{{\mathcal P}} \nc{\bbE}{\mathbb{E}}
\nc{\bbEQ}{\mathbb{E}_{\mathbb{Q}}} \nc{\eps}{\varepsilon}
\nc{\bbEP}{\mathbb{E}_{\mathbb{P}}}\nc{\bbL}{\mathbb{L}}
\nc{\bbP}{\mathbb{P}} \nc{\bbQ}{\mathbb{Q}} \nc{\del}{\partial}
\nc{\Om}{\Omega} \nc{\om}{\omega} \nc{\bbR}{\mathbb{R}}
\nc{\bbC}{\mathbb{C}} \nc{\bfr}{\begin{flushright}}
\nc{\efr}{\end{flushright}} \nc{\dXt}{\Delta X_{t}} \nc{\dXs}{\Delta
X_{s}} \nc{\bs}{\blacksquare} \nc{\dX}{\Delta X} \nc{\dY}{\Delta Y}
\nc{\dnkx}{\left(X(T^{n}_{k})-X(T^{n}_{k-1})\right)}
\nc{\esssup}{\mathrm{ess}\mbox{ }\mathrm{sup}}
\nc{\essinf}{\mathrm{ess}\mbox{ } \mathrm{inf}}
\nc{\dhats}{\widehat{\delta_s}} \nc{\tX}{\tilde{X}}
\nc{\tZ}{\tilde{Z}}
\nc{\what}{\widehat}
 \nc{\half}{\frac{1}{2}}

\def\rar{\rightarrow} \nc{\uar}{\uparrow}

\nc{\chf}{\mbox{$\mathbf1$}} \nc{\eid}{\stackrel{d}{=}}

\begin{document}
\title{Linear inverse problems for Markov processes and their regularisation}
\author{Umut \c{C}etin}
\address{Department of Statistics, London School of Economics and Political Science, 10 Houghton st, London, WC2A 2AE, UK}
\email{u.cetin@lse.ac.uk}
\date{\today}
\begin{abstract}
We study the solutions of the inverse problem
\[
g(z)=\int f(y) P_T(z,dy) 
\]
for a given $g$, where $(P_t(\cdot,\cdot))_{t \geq 0}$ is the transition function of a given symmetric Markov process, $X$, and $T$ is a fixed deterministic time, which is linked to the solutions of the ill-posed Cauchy problem
\[
u_t + A u=0, \qquad u(0,\cdot)=g,
\]
where $A$ is the generator of $X$. A necessary and sufficient condition ensuring square integrable solutions is given. Moreover, a family of regularisations for  above  problems is suggested. We show in particular that these inverse problems have a solution when $X$ is replaced by  $\xi X + (1-\xi)J$, where $\xi$ is a Bernoulli random variable, whose probability of success can be chosen arbitrarily close to $1$, and $J$ is a suitably constructed jump process.
\end{abstract}
\maketitle

\section{Introduction}
Suppose that $X$ is a Markov process taking values in some  topological space, $\bfE$,  and let $(P_t)_{t \geq 0}$ be a strongly continuous semigroup describing the movement of $X$ in its state space through time.  Let us consider the following integral equation 
\be \label{e:IP}
g(z)=\int f(y) P_T(z,dy) 
\ee
for a given $g$ and a fixed -deterministic- $T \geq 0$.  Put differently, the above can be viewed as recovering an input signal, $f$, from a blurred output, $g$, which is corrupted by some noise described by the kernel $P_T$.   This is an inverse problem which is present in many fields of science and technology. In image processing solving this inverse problem corresponds to  the reconstruction of an image from the available data as in, e.g., {\em tomography} (see \cite{imaging}). In statistics one is often interested in estimating the density function, $f$, of a certain random variable using the observations of a related one with density $g$, which is linked by some kernel $K$  to the original density via the equation $g= Kf$.  Vardi and Lee \cite{VL} show that such inverse problems can be interpreted as a statistical estimation problem from an incomplete data if it admits a positive solution. Under the assumption of the existence of a positive solution to $g=Kf$ they develop a maximum likelihood (ML) algorithm  to solve the estimation issue and apply their methodology to problems arising from optimal investment, emission tomography, and image reconstruction due to motion blurring. More recent works on the interplay between ML estimators and inverse problems with positivity constraints include \cite{Silvermanetal}, \cite{eggL}, and \cite{KooChung}, and \cite{Cavalier} is an excellent introduction to inverse problems in statistics and a survey of available methods.  Note that the inverse problem given by $g=Kf$, where $K$ is a non-negative operator on a Hilbert space with norm less than $1$, can be recast in the form of (\ref{e:IP}). Indeed, if we define the operator $A:=-\log K$ (consult the beginning of the next section or Chapter 1 of \cite{FOT} for the construction of this operator), then $A$ will correspond to the infinitesimal generator of a Markov process whose transition function at time 1 coincides with $K$, i.e $g=Kf$ becomes $g(z)=\int f(y)P_1(z,dy)$, where $(P_t)_{t\geq 0}$ is the semigroup of the Markov process with generator $-\log K$. Thus, the method that we shall describe below will contain as special cases  many inverse problems in the literature and, in particular,  the above density estimation problem of statistics. Moreover, the existing literature typically assumes that  $K$ is a compact operator to arrive at a simple  {\em singular value decomposition}. We will not need this assumption in what follows and thereby considerably extend the scope of the methodology for solving inverse problems.

The inverse problem in (\ref{e:IP}) has an alternative partial differential equation (PDE) interpretation. Suppose that for a given $g$ and fixed $T>0$ one can find a solution, $f$, to (\ref{e:IP}). Then, one can  easily show that $u(t,\cdot):=P_{T-t}f$ is a solution to  the following:
\be \label{i:e:pde}
u_t + A u=0, \qquad u(0,\cdot)=g,
\ee
where $A$ is the generator of $X$. If $A$ is a differential operator, the above is a {\em backward} PDE with an {\em initial} condition. Such equations are known to be  ill-posed in the sense of Hadamard that either there exists no solution or the solution is non-unique, or the unique solution does not have a continuous dependence on the initial data, $g$. 

If, moreover, $g$ is a probability density then  the  inverse problem (\ref{e:IP}) with positivity constraint can answer  the following question: Can we find an initial distribution for $X$ so that the probability distribution of $X_T$ is defined by $g$? As such, this question is  related to the {\em Skorokhod embedding} problem which searches for a martingale whose time $T$-distribution is given by $g$.  Ekstr\"om et al. \cite{Ekstrom} have solved this problem of Skorokhod in the one-dimensional setting  by showing the existence of a generalised diffusion with a  constant initial value, which is set from the beginning as a consequence of the martingale condition. However, the transition function of this martingale is in general cannot be obtained. The inverse problem (\ref{e:IP}), on the other hand, fixes the transition function from the start and seeks an initial distribution rather than a whole stochastic process. There is nevertheless an important drawback: Although we may obtain arguably quite explicit answers using  the positive solutions of (\ref{e:IP}), one cannot expect to find a solution, let alone a positive one, to this equation for any given distribution $g$ in general. The reason for this is that the inverse of $P_t$, $P_t^{-1}$,  is typically an unbounded linear operator and, therefore, it has only a dense domain. However, our methodology is not restricted to the one-dimensional case and works equally effective in a multi-dimensional setting.

In  what follows we aim to find  necessary and sufficient conditions for (\ref{e:IP}) to admit a suitably integrable solution. There do not seem to be many attempts in the literature to characterise the solutions of such an inversion. The first attempt when $(P_t)_{tX \geq 0}$ is the transition function of a linear Brownian motion  is by Widder \cite{WidderTams}. In \cite{WidderTams} and some subsequent works  Widder provides some necessary and sufficient conditions for the existence of a  solution to  this inverse problem, which he calls {\em Weierstrass transform}. 

We show in Theorem \ref{t:invformula} that (\ref{e:IP}) has a square integrable  solution if and only if
\[
\int_0^{\infty}I_0(2\sqrt{2Tt})\int_0^{\infty}J_0(2 \sqrt{ts})e^{-\alpha s} (P_sg,g) ds dt <\infty,
\]
where $J_0$ (resp. $I_0$) is the (resp. modified) Bessel function of the first kind of order $0$. Additionally, the same theorem  gives a formula for the inversion. Section 2 also contains a number of alternative criteria for the characterisation of the domain of $P_t^{-1}$.  In particular it is observed that the finiteness of the double integral above can be recast in terms of the last passage times in the case of one-dimensional regular diffusions.  Moreover, Corollary \ref{c:picard} gives  us a numerical recipe by means of a Picard iteration to deduce the convergence of this integral. 

As we mentioned earlier there is no solution to (\ref{e:IP}) in general for an arbitrary transition function. Theorem \ref{t:regularisation} introduces a family of regularisations of (\ref{e:IP}), which are essentially small perturbations of the original problem aimed at obtaining a solution for any given $g$. Moreover, the solution of the regularised problem is characterised in terms of the minimiser of an associated optimisation problem.

Corollary \ref{c:mixing} gives a remarkable special class of regularisations suggested by Theorem \ref{t:regularisation}. It shows that if we construct a new Markov process that amounts to choosing randomly between the original process, $X$, and a suitable pure jump process, then the inverse problem will have a solution for every $g$ as soon as $P_T$ is replaced by the corresponding transition function of the new Markov process. For example, when $X$ is a Brownian motion, the inverse problem can be turned into a well-posed one by replacing the Brownian motion with a mixture of a Brownian motion and a compound Poisson process whose jumps are normally distributed. Such mixtures of the original Markov process and a jump process are easy to construct and one can choose the probability of choosing the jump process arbitrarily small so that the jump component is virtually absent in implementations. This mixture also regularises the ill-posed backward PDE (\ref{i:e:pde}) by transforming it to a partial integro-differential equation using an arbitrarily small perturbation.

Although we are able to give a necessary and sufficient condition for the existence of a solution to (\ref{e:IP}), what is particularly missing in this paper is a comparison result. Namely, if we know that $P_T^{-1}h$ exists for some $h$, what kind of relationship between $g$ and $h$ would entail that $g$ is also in the domain of $P_T^{-1}$? A comparison theorem in the spirit of the ones that can be found in the literature on the {\em Skorokhod embedding problem} could be very useful.  Falkner \cite{falkner} has shown (under a duality assumption and another mild condition) for a general transient Markov process, $X$, with potential operator $U$ that if $U\mu \leq U\nu$ for measures $\mu$ and $\nu$, then one can find a stopping time $\tau$ such that $X_{\tau}$ has the law $\mu$ if $\nu$ is the distribution of $X_0$.  Note that in order for $g$ to be in the domain of  $P_T^{-1}$ it is necessary that $Ug \leq  Uh$ for some $h$ in the domain of $P_T^{-1}$.
However, the following counterexample\footnote{This example is due to David Hobson.} shows that this necessary condition is not sufficient: Let $g$ be the distribution of $X_{\tau}$, where $\tau = \inf\{t \geq T : |X_t| > a\}$ and $X$
is a standard Brownian motion with $X_0=0$ and killed as soon as it exits $(-2a, 2a)$. Clearly, $Ug \leq U\eps_0$, where $\eps_0$ is the Dirac measure at $0$. However, $g$ cannot be in the domain of $P_T^{-1}$ since it has a point
mass. 

The outline of the paper is as follows. Section 2 presents the setup and introduces the inverse problem. It contains Theorem \ref{t:invformula} that gives the necessary and sufficient condition for the inversion along with the inversion formula. Section 3 is devoted to the regularisation of the inverse problem and includes in particular Corollary \ref{c:mixing}, which states that the inverse problem has a solution when $X$ is replaced by a mixture of $X$ and a jump process. 
\section{An inverse problem for a symmetric Markov process} \label{s:ip}
Let us fix a Borel right Markov process $X=(\Om, \cF,\cF_t,X_t,\theta_t, P^x)$ with lifetime $\zeta$, state space $(\bfE, \sE)$, sub-Markovian semigroup $(P_t)$, and resolvent $(U^{\alpha})$. Suppose that  $\bfE$ is a locally compact separable metric space and $(P_t)$ is $m$-symmetric with respect to a $\sigma$-finite measure $m$ on the Borel $\sigma$-algebra $\sE$ with $\mbox{supp}[m]=\bfE$. More precisely, we assume that $(P_t)$ can be extended to a strongly continuous sub-Markovian semigroup on $L^2(\bfE,m)$ such that 
\[
(P_tf,g)=(f,P_tg), \qquad \forall f,g \in L^2(m),
\]
where $(\cdot, \cdot)$ denotes the inner product with respect to $m$, i.e. $(f,g)=\int_{\bfE}fgdm$ for $f,g \in L^2(\bfE,m)$.  We also assume that  $(\bfE,\sE,m)$ is a separable measure space, which in turn implies that $L^2(\bfE,m)$ is a separable Hilbert space. In the sequel we shall simply write $L^2(m)$ instead of $L^2(\bfE,m)$.

The generator, $A$, of $(P_t)$ is defined as usual via
\bea 
A f&=&\lim_{t \rar 0} \frac{P_tf -f}{t} \label{e:defA}\\
\sD(A)&:=& \{ f \in L^2(m): \mbox{The limit } (\ref{e:defA}) \mbox{ exists in } L^2(m)\}. \nn
\eea

Consequently $-A$ is a non-negative definite symmetric operator on $L^2(m)$. Thus, there exists a spectral family\footnote{That is, 1) $E_{\lambda}E_{\mu}=E_{\lambda}, \, \lambda \leq \mu$; 2) $\lambda \mapsto E_{\lambda}f$ is right continuous for any $f \in L^2(m)$; and  3) $\lim_{\lambda \rar \infty}E_{\lambda}f=f$ for all $f \in L^2(m)$. In particular $(E_{\lambda}f,g)$ is of bounded variation in $\lambda$ for any $f,g  \in L^2(m)$.} $\{E_{\lambda}:0\leq \lambda <\infty\}$ of projection operators such that $-A=\int_0^{\infty}\lambda dE_{\lambda}$. This further  entails
\[
\sD(A)=\left\{f \in L^2(m):\int_0^{\infty}\lambda^2 d(E_{\lambda}f,f)<\infty\right\}.
\]
Moreover, if $\phi:\bbR_+ \mapsto \bbR$ is a continuous function, $\phi(-A)$ is another symmetric operator on $L^2(m)$ with the spectral representation $\int_0^{\infty}\phi(\lambda) dE_{\lambda}$ and domain
\[
\sD(\phi(-A))=\left\{f \in L^2(m):\int_0^{\infty}\phi(\lambda)^2 d(E_{\lambda}f,f)<\infty\right\}.
\]
In particular, for each $t>0$ and $\alpha>0$, $P_t=\int_0^{\infty}e^{-\lambda t} dE_{\lambda}$ and $U^{\alpha}=\int_0^{\infty}\frac{1}{\lambda+\alpha}dE_{\lambda}$, and obviously, have the whole $L^2(m)$ as their domain. When $X$ is transient, the potential operator is also given by $U=\int_0^{\infty}\frac{1}{\lambda}dE_{\lambda}$. We refer the reader to Chapter 1 of \cite{FOT} for a brief account of the spectral theory associated with the  generators of strongly continuous symmetric semigroups. 
\begin{example} \label{ex:1dim} Let $\bfE=(l,r)$ with $-\infty \leq l<r \leq \infty$ and consider a one-dimensional regular diffusion on natural scale defined by the generator
\[
Af =\frac{\half d \frac{df}{dx}-udk}{dm},
\]
where the {\em killing measure}, $k$, and the {\em speed measure}, $m$, are Radon measures on $(\bfE, \sE)$.  In the absolutely continuous case the generator becomes a differential operator:
\[
Af =\frac{\sigma^2}{2}f'' - cf,
\]
where $\sigma : \bfE \mapsto \bbR_{++}$ and $c:\bfE:\mapsto \bbR_+$ are measurable functions. 

McKean \cite{MK}  has shown that the transition function, $(P_t)$ possesses a symmetric density, $(p(t,\cdot,\cdot)$, with respect to $m$ such that
\[
p(t,x,y)=\int_0^{\infty}e^{-\lambda t} e(x,y,d\lambda),
\]
where $e(x,y,\cdot)$ is a measure on $[0,\infty)$ with $e(x,y\cdot)=e(y,x,\cdot)$. We refer the reader to \cite{MK} or \cite{IK} for more details on the general theory of one-dimensional diffusions and the eigendifferential expansions of their transition densities. 

When the diffusion has no natural boundaries Elliott \cite{jelliott} has shown earlier that the spectrum of the generator is discrete, which in turn implies that the transition density with respect to $m$ is given by
\[
p(t,x,y) =\sum_{n=0}^{\infty} e^{-\lambda_n t} \phi_n(x)\phi_n(y),
\]
where $0\leq \lambda_0 \leq \lambda_1 \leq \ldots \lambda_n \leq \ldots$ with $\lambda_n \uparrow \infty$ and $\phi_n$ is the solution of $A\phi_n = -\lambda_n \phi_n$   for appropriate boundary conditions given by the behaviour of the diffusion near $l$ and $r$. 
\end{example}
\begin{example} \label{ex:jump} Suppose that $q(x,y)=q(y,x)$ and $qdm$ defines a transition function on $(\bfE, \sE)$. In particular,
\[
\int_{\bfE}q(x,y) m(dy)\leq 1.
\]
Then, it can be directly verified that $A$ defined by
\[
Af(x) =\int_{\bfE} f(y)q(x,y)m(dy) -f(x)
\]
is a bounded symmetric operator corresponding to a Markov jump process (see, e.g., Section 4.2 in \cite{EK}) that remains constant between  jumps of a Poisson process with unit intensity and moves between the states of $\bfE$ according to the kernel $q$, or is being sent to the cemetery state with probability $1-\int_{\bfE}q(x,y) m(dy)$.

In the particular case of $\bfE=\bbR$, $q(x,y)=q(y-x)$ for some symmetric function, $q$, $A$ is the generator of a compound Poisson process whose jumps have a symmetric distribution around $0$ with $m$-density $q$, provided
\[
\int_{-\infty}^{\infty} q(x,y)m(dy)=1, \; \forall x\in \bbR.
\]
\end{example}
Next consider the inverse operator, $P_{t}^{-1}$, for $t>0$. That is,
$g \in \sD(P_t^{-1})$ if $g \in L^2(m)$ and there exists $f \in L^2(m)$ such that $P_t f=g$. In this case we shall define $P_t^{-1}g$ to be $f$. Note that this operation is well-defined. Indeed, if $f_1, f_2 \in L^2(m)$ are such that $g=P_tf_1=P_tf_2$, then $P_t(f_1-f_2)=0$. However, in view of the spectral representation of $P_t$, this yields $(E_{\lambda}(f_1-f_2), f_1-f_2)=0$ for all $\lambda \geq 0$, which in turn implies $f_1=f_2$, $m$-a.e. due to the fact that $\lim_{\lambda\rar \infty}(E_{\lambda}f,f)=(f,f)$ for any $f \in L^2(m)$.

Observe that, since $P_t$ is a bounded self-adjoint operator, $P_t^{-1}$ is also a symmetric operator on $L^2(m)$. The following, which should formally follow from spectral calculus, characterises $P_t^{-1}$ in terms of the spectral family $(E_{\lambda})$.
\begin{theorem}\label{t:invchar} Let $P_t^{-1}$ be the inverse of  $P_t$ for $t>0$. Then the following hold. 
\bean
\sD(P_t^{-1})&=&\left\{g \in L^2(m):\int_0^{\infty}e^{2 \lambda t} d(E_{\lambda}g,g)<\infty\right\} \\
P_t^{-1}g&=&  \int_0^{\infty}e^{\lambda t} dE_{\lambda}g. \label{e:invrep}
\eean
\end{theorem}
\begin{proof}
Let 
\[
D=\left\{g \in L^2(m):\int_0^{\infty}e^{2 \lambda t} d(E_{\lambda}g,g)<\infty\right\}.
\]
This is clearly the domain of the operator $I_t$ on $L^2(m)$, where
\[
I_t=\int_0^{\infty}e^{\lambda t} dE_{\lambda}.
\]
Pick an arbitrary $g \in \sD(P_t^{-1})$. By definition there exists $f \in L^2(m)$ such that $g=P_tf$. Using the spectral representation of the semigroup we may write
\[
g=\int_0^{\infty}e^{-\lambda t} dE_{\lambda}f.
\]
Thus,
\bean
\int_0^{\infty}e^{2 \lambda t} d(E_{\lambda}g,g)&=&\int_0^{\infty}e^{2 \lambda t} d(E_{\lambda}P_tf,P_tf)=\int_0^{\infty}e^{2 \lambda t} d(E_{\lambda}P_{2t}f,f)\\
&=&\int_0^{\infty} d(E_{\lambda}f,f)=\|f\|^2<\infty,
\eean
where the second equality follows from the symmetry of $P_t$, $dE_{\lambda}P_tf=e^{-\lambda t}dE_{\lambda}f$, and the fact that $E_{\lambda}$ and $P_t$ commute. Thus, $\sD(P_t^{-1}) \subset D$.  Moreover,
\[
I_t g = \int_0^{\infty}e^{\lambda t} dE_{\lambda}g= \int_0^{\infty}e^{\lambda t} dE_{\lambda}P_tf =\int_0^{\infty} dE_{\lambda}f =f,
\]
i.e, $P_t^{-1}=I_t$ on $\sD(P_t^{-1})$.

Thus, it remains to show that $D \subset \sD(P_t^{-1})$. Indeed, let $g \in D$ and set $f =I_t g$. Note that $f \in L^2(m)$ by the definition of $D$. Moreover, $dE_{\lambda }f=e^{\lambda t} dE_{\lambda }g$. Therefore,
\[
P_t f =\int_0^{\infty}e^{-\lambda t} dE_{\lambda }f=\int_0^{\infty} dE_{\lambda }g=g.
\]
Hence, $D \subset \sD(P_t^{-1})$.
\end{proof}
The above result illustrates the first difficulty with inverting $P_t$. When $A$ is an unbounded operator, which is usually the case, so is $P_t^{-1}$. In this case $P_t^{-1}$ will have a dense domain, characterisation of which is one of the main goals of this paper. 

On the other hand, if $A$ is  bounded,  $E_{\lambda}$ becomes the identity operator for all $\lambda\geq M$ for some $M<\infty$.  In view of the above representation for $P_t^{-1}$ and its domain, this boundedness property will be inherited by $P_{t}^{-1}$. 
\begin{corollary} Suppose that the generator, $A$, of $(P_t)$ is bounded. Then $\sD(P_t^{-1})=L^2(m).$ In particular (\ref{e:invrep}) holds for all $g \in L^2(m)$.
\end{corollary}
\begin{remark} \label{r:positive} It is tempting to conclude that $P_t^{-1}g$ is nonnegative when $g\geq 0$ and belongs to $\cD(P_t^{-1})$. This would be especially handy when one needs to estimate the true density $f$ by observing an auxiliary density $g$ using the relationship $g=P_t f$.  However, the positivity of $f$  does not in general hold although one can find instances in the literature (see, e.g., the beginning of Section 3.3 in \cite{KooChung}) where this issue is overlooked. 

To see this in a concrete example suppose that $X$ is an Ornstein-Uhlenbeck process, i.e
\[
X_t=X_0 + B_t -r\int_0^t X_s ds, \qquad r>0.
\]
Then, conditional on $X_0=x$, $X_t$ is normally distributed with mean $xe^{rt}$ and variance $\frac{1-e^{-2rt}}{2r}$. The speed measure for this diffusion is given by
\[
m(dx) =e^{-rx^2}dx,
\]
thus its generator, $A$, is symmetric with respect to $m$. Then, if one takes $g=x^2$, it follows from a simple computation that $g=P_1f$, where 
\[
f(x)= e^{2r}x^2 -\frac{e^{2r}-1}{2r}.
\]
Note that both $f$ and $g$ belong to $L^2(m)$. However, $f$ is not always nonnegative on $\bbR$.
\end{remark}

As mentioned in Introduction the inverse problem (\ref{e:IP}) is intimately linked to the solution of a Cauchy problem, which becomes a backward partial differential equation when $A$ is a differential operator.
\begin{corollary} \label{c:invpde} The following hold for any fixed $T>0$.
\begin{enumerate} 
\item Suppose that $g \in \sD(P_T^{-1})$. Then there exist $(u(t,\cdot))_{t \in [0,T]}$ such that $u(t,\cdot) \in L^2(m)$ for all $t \in [0,T]$, and
\be \label{e:bpde}
u_t+Au=0, \; t>0, \mbox{ and } \quad u(0, \cdot)=g,
\ee
where $u_t(t, \cdot):=\lim_{h \rar 0} \frac{u(t+h,\cdot)-u(t,\cdot)}{h}$ and the  limit is in $L^2(m)$. 
\item Conversely, if there exists a family  $(u(t,\cdot))_{t \in [0,T]}\subset L^2(m)$ solving (\ref{e:bpde}) for a given $g \in L^2(m)$, then $g \in \sD(P_T^{-1})$. 
\end{enumerate}
Consequently, there exists a unique solution of (\ref{e:bpde}) in $L^2(m)$ if and only if $g \in \sD(P_T^{-1})$. Moreover, $P_T^{-1}g=u(T,\cdot)$.

\end{corollary}
\begin{proof}
Let $f=P_T^{-1}g$ and define $u(t,\cdot)=P_{T-t}f$. First observe that since $f \in L^2(m)$, $P_tf \in \sD(A)$ for all $t>0$. Indeed,
\[
\int_0^{\infty}\lambda^2d(E_{\lambda}P_tf, P_tf)=\int_0^{\infty}\lambda^2e^{-2\lambda t} d(E_{\lambda}f, f)\leq \frac{1}{t^2 e^2}\int_0^{\infty}d(E_{\lambda}f, f)<\infty.
\]
Thus,
\[
AP_{T-t}f=-\int_0^{\infty}\lambda e^{-(T-t)\lambda}dE_{\lambda}f.
\]
Moreover, 
\[
\frac{d}{dt}P_{T-t}f=\frac{d}{dt}\int_0^{\infty} e^{-(T-t)\lambda}dE_{\lambda}f=\int_0^{\infty} \lambda e^{-(T-t)\lambda}dE_{\lambda}f
\]
by virtue of the dominated convergence theorem since $(E_{\lambda}f, f)$ is of bounded variation in $\lambda$ and $x^2e^{-2x}$ is bounded on $[0,\infty)$. Therefore, $u$ solves (\ref{e:bpde}) since $u(0,\cdot)= P_Tf=g$. 

Conversely, suppose $u$ is a solution of (\ref{e:bpde}) in $L^2(m)$. In particular, $u(t, \cdot) \in \sD(A)$ for $t \in (0,T]$.  Thus for any $t \in (0,T]$, we have 
\be \label{e:uevolve}
u(t,\cdot)=g + \int_0^t u_t(s,\cdot)ds= g - \int_0^s ds A u(s,\cdot),
\ee
where the integrals are to be understood as Bochner integrals in $L^2(m)$. 

Next observe that for any $\lambda \geq 0$ and $f \in L^2(m)$,
\[
E_{\lambda}f = \int_0^{\lambda} dE_{\mu}f \in \sD(A),
\]
and 
\[
AE_{\lambda}f=-\int_0^{\lambda}\mu dE_{\mu}f.
\]
Applying $E_{\lambda}$ to both sides of (\ref{e:uevolve}) and exploiting the commutativity of $E_{\lambda}$ and $A$ we  obtain
\[
E_{\lambda}u(t,\cdot)= E_{\lambda} g +\int_0^t ds \int_0^{\lambda}\mu dE_{\mu}u(s,\cdot).
\]
However, the unique solution of the above equation is given by
\[
E_{\lambda}u(t,\cdot)=\int_0^{\lambda}e^{\mu t} dE_{\mu}g,
\]
which readily yields $dE_{\lambda}u(t,\cdot)=e^{\lambda t} dE_{\lambda}g$.
Therefore,
\[
P_T u(T,\cdot)= \int_0^{\infty} e^{-\lambda T} dE_{\lambda} u(T,\cdot)=g.
\]
Since $u(T,\cdot) \in L^2(m)$, we deduce that $g \in \sD(P_T^{-1})$.

Thus, we have shown that there is a one-to-one  correspondence between $\sD(P_{T}^{-1})$ and the $L^2$-solutions of (\ref{e:bpde}). Moreover, since $P_{T}^{-1}g$ is uniquely determined, any solution of (\ref{e:bpde}) satisfies $u(T,\cdot)=P_{T}^{-1}g$.

Finally, by virtue of $dE_{\lambda}u(t,\cdot)=e^{\lambda t} dE_{\lambda}g$  we readily establish the uniqueness of $L^2$-solutions of (\ref{e:bpde}) under the assumption that $g \in \sD(P_T^{-1})$.
\end{proof}
Theorem \ref{t:invchar} characterises the domain of $P_t^{-1}$ completely. However, it requires the knowledge of the spectral resolution. Theorem \ref{t:invformula}, on the other hand,  determines the domain of $P_T^{-1}$ in terms of the transition function. Before its statement let us introduce a new operator on $L^2(m)$:
\be \label{e:Jdef}
\cJ^{\alpha}_t f := \int_0^{\infty}J_0(2 \sqrt{ts})e^{-\alpha s} P_sf ds,
\ee
where $\alpha >0$, $J_0$ is the Bessel function of the first kind of order $0$, and the integral is to be understood as a Bochner integral. Since $J_0$ is bounded and $U^{\alpha}$ is a bounded operator, it follows that $\cJ^{\alpha}_t$ is also a bounded operator and, thus, has $L^2(m)$ as its domain. 
\begin{proposition} \label{p:J} Let $(\cJ^{\alpha}_t)$ be the family of operators defined by (\ref{e:Jdef}). For each $t>0$ and $\alpha>0$ $\cJ^{\alpha}$ is a non-negative self-adjoint operator on $L^2(m)$ with the following spectral resolution:
\be \label{e:Jspec}
\cJ^{\alpha}_t  =\int_0^{\infty} \frac{1}{\lambda+\alpha}e^{-\frac{t}{\lambda+\alpha}}dE_{\lambda}.
\ee
Moreover, for any $f \in L^2(m)$, the mapping $t \mapsto (\cJ_t^{\alpha} f, f)$ is convex in $t$ and decreases to $0$ as $t \rar \infty$.
\end{proposition}
\begin{proof}
Let us first show that (\ref{e:Jspec}) holds. Indeed, using Fubini and the spectral representation of $P_t$ along with the fact that $\int_0^{\infty} e^{-\alpha s} J_0(2\sqrt{ts})ds=e^{-\frac{t}{\alpha}}/\alpha$ (see Table 29.2 in \cite{AS}), we obtain
\[
(\cJ_t^{\alpha}f,g)=\int_0^{\infty}\left(\int_0^{\infty} J_0(2\sqrt{ts})e^{-(\alpha+\lambda) s}ds \right)d(E_{\lambda} f,g)=\int_0^{\infty} \frac{1}{\lambda+\alpha}e^{-\frac{t}{\lambda+\alpha}}d(E_{\lambda}f,g),
\]
which yields (\ref{e:Jspec}). Thus,
\[
(\cJ_t^{\alpha}f,f)=\int_0^{\infty} \frac{1}{\lambda+\alpha}e^{-\frac{t}{\lambda+\alpha}}d(E_{\lambda}f,f)\geq 0
\]
since $E_{\lambda}$ is a non-negative operator. It  can be checked directly  that $\cJ_t^{\alpha}$ is symmetric, and therefore self-adjoint due to its boundedness. 

The spectral representation also yields the monotonicity and the convexity of the map $t \mapsto (\cJ_t^{\alpha} g, g)$. The fact that $\lim_{t \rar \infty}(\cJ_t^{\alpha} f, f) =0$ is a consequence of the monotone convergence theorem and the assumption that $f \in L^2(m)$.
\end{proof}
\begin{theorem} \label{t:invformula}
$g \in \sD(P_t^{-1})$ if and only if 
\[
\int_0^{\infty}I_0(2\sqrt{2ts})(\cJ^{\alpha}_s g,g) ds <\infty,
\]
where $I_0$ is the modified Bessel function of the first kind of order $0$. Moreover, if $g \in \sD(P_t^{-1})$, then $P_t^{-1}g$ equals a Bochner integral as follows:
\be \label{e:invformula}
P_t^{-1}g=e^{-\alpha t}\int_0^{\infty}I_0(2\sqrt{ts})\cJ^{\alpha}_sgds.
\ee
\end{theorem}
\begin{proof}
 It follows from (\ref{e:Jspec}) that
 \bean
\int_0^{\infty}I_0(2\sqrt{2ts})(\cJ^{\alpha}_s g,g) ds &=&\int_0^{\infty} ds I_0(2\sqrt{2ts}) \int_0^{\infty} \frac{1}{\lambda+\alpha}e^{-\frac{s}{\lambda+\alpha}}d(E_{\lambda}g,g)\\
&=&\int_0^{\infty} \left(\int_0^{\infty}e^{-\frac{s}{\lambda+\alpha}} I_0(2\sqrt{2ts})  ds\right)\frac{1}{\lambda+\alpha}d(E_{\lambda}g,g)\\
&=&\int_0^{\infty} e^{2t(\lambda+\alpha)} d(E_{\lambda}g,g),
\eean
which is finite if and only if $g \in \sD(P_T^{-1})$. The last line in the above follows from the Laplace transform of the modified Bessel function (see Table 29.3 in \cite{AS}).

Next observe that for $g \in L^2(m)$,
\[
\|\cJ^{\alpha}_t g\|^2=\int_0^{\infty}\frac{1}{(\lambda+\alpha)^2}e^{-\frac{2t}{\lambda +\alpha}}d(E_{\lambda}g,g)\leq\frac{1}{e^2t^2} \int_0^{\infty}d(E_{\lambda}g,g)=\frac{\|g\|^2}{e^2t^2}.
\]
Thus, using Fubini's theorem and (\ref{e:Jspec}) we deduce
\[
\int_0^{\infty}I_0(2\sqrt{ts})\cJ^{\alpha}_sgds=e^{\alpha t}\int_0^{\infty} e^{t\lambda} dE_{\lambda}g,
\]
which implies (\ref{e:invformula}).
\end{proof}
$\cJ^{\alpha}g$ can be explicitly computed if one knows the transition function of $X$. If one instead has the knowledge of the family $(U^{\alpha})$, $\cJ^{\alpha}g$ is determined as the solution of a Cauchy problem. 
\begin{theorem} \label{t:Jcauchy}
Given an $f \in L^2(m)$ there exists a unique solution to the following Cauchy problem:
\bea \label{e:Jcauchy}
\frac{d}{dt}j(t,\cdot) &=& -U^{\alpha}j(t,\cdot) \\
j(0,\cdot)& =& U^{\alpha}f. \nn
\eea
Moreover, its solution is given by $j(t,\cdot)=\cJ^{\alpha}_tf$.
\end{theorem}
\begin{proof}
Let $j(t,\cdot)=\cJ^{\alpha}_tf$ and observe using (\ref{e:Jspec}) that
\[
j(t,\cdot)=\int_0^{\infty}\frac{1}{\lambda+\alpha}e^{-\frac{t}{\lambda+\alpha}}dE_{\lambda}f.
\]
Thus,
\[
U^{\alpha}j(t,\cdot)=\int_0^{\infty}\frac{1}{(\lambda+\alpha)^2}e^{-\frac{t}{\lambda+\alpha}}dE_{\lambda}f.
\]
In view of the Fubini's theorem
\[
\int_0^t U^{\alpha}j(s,\cdot)ds=\int_0^{\infty}\frac{1}{\lambda+\alpha}\left(1-e^{-\frac{t}{\lambda+\alpha}}
\right)dE_{\lambda}f=U^{\alpha}f-j(t,\cdot),
\]
which verifies that $\cJ^{\alpha}_tf$ solves (\ref{e:Jcauchy}) since 
\[
\cJ^{\alpha}_0f=\int_0^{\infty}\frac{1}{\lambda+\alpha}dE_{\lambda}f=U^{\alpha}f.
\]
To show the uniqueness let us suppose $j_1$ and $j_2$ are two solutions of (\ref{e:Jcauchy}) and set $j=j_1-j_2$. 
Note that
\[
U^{\alpha}E_{\lambda}j(t,\cdot)= \int_0^{\lambda}\frac{1}{\mu+\alpha}dE_{\mu}j(t,\cdot).
\]
Since $j$ solves (\ref{e:Jcauchy}) with the initial condition $0$, applying $E_{\lambda}$ to both sides of the equality we obtain 
\[
E_{\lambda}j(t,\cdot)=-\int_0^t ds\int_0^{\lambda}\frac{1}{\mu+\alpha}dE_{\mu}j(s,\cdot),
\]
which yields $E_{\lambda}j(t,\cdot)=0$ for all $\lambda \geq 0$. This completes the proof.
\end{proof} 
Since $-U^{\alpha}$ is a non-positive bounded operator, it generates a uniformly continuous semi-group, $T_t:=e^{-t U^{\alpha}}$. Thus, we have the following immediate corollary.
\begin{corollary} \label{c:picard} Let $(T_t)$ be the semigroup on $L^2(m)$ generated by $U^{\alpha}$. Then,
\[
\cJ^{\alpha}_t f= T_t U^{\alpha}f.
\]
\end{corollary}
The fact that $\cJ^{\alpha}_tf$ is a solution of a Cauchy problem with a bounded generator also implies that one can compute  it using a Picard iteration.
\begin{corollary} Suppose $f \in L^2(m)$ and set
\bean
j_0(t,\cdot)&=&U^{\alpha }f,\\
j_{n+1}(t,\cdot)&=&U^{\alpha} f-\int_0^t U^{\alpha}j_n(s,\cdot)ds.
\eean
Then, $(j_n(\cdot,\cdot))_{n \geq 0}$ converges uniformly in $L^2( m)$ to $(\cJ^{\alpha}_sf)_{s \in [0,t]}$ for any $t>0$, i.e. 
\[
\lim_{n \rar \infty} \sup_{0\leq s \leq t} \|j_{n}(s,\cdot)-\cJ^{\alpha}_sf\|=0, \qquad \forall t>0.
\]
\end{corollary}
\begin{proof} 
Let $j(s,\cdot)=\cJ^{\alpha}_sf$ and observe from (\ref{e:Jcauchy}) that 
\[
j(s,\cdot)=U^{\alpha}f -\int_0^s U^{\alpha}j(r,\cdot)dr.
\]
Moreover,
\bea 
\|j_0(s,\cdot)-j(s,\cdot)\| &\leq & \int_0^s \| U^{\alpha}j(r,\cdot)\|dr = \int_0^s \sqrt{\int_0^{\infty} \frac{1}{(\lambda+\alpha)^2}e^{-\frac{2r}{\lambda +\alpha}}d(E_{\lambda}f,f)}dr \nn\\
&\leq&\int_0^s \frac{e^{- \frac{r}{\alpha}}}{\alpha}\|f\|dr \leq \|f\|. \label{e:picest0}
\eea
Thus,
\[
 \|j_{n+1}(s,\cdot)-j(s,\cdot)\|\leq\int_0^s\|U^{\alpha}j_n(r,\cdot)-U^{\alpha}j(r,\cdot)\|dr\leq \frac{1}{\alpha} \int_0^s\|j_n(r,\cdot)-j(r,\cdot)\|dr.
\]
Hence, 
\[
\sup_{0\leq s \leq t} \|j_{n+1}(s,\cdot)-j(s,\cdot)\|\leq \frac{1}{\alpha} \int_0^t \|j_n(r,\cdot)-j(r,\cdot)\|dr
\]
and we deduce by induction that
\[
\sup_{0\leq s \leq t} \|j_{n}(s,\cdot)-j(s,\cdot)\|\leq \frac{t^n	}{\alpha^n n!} \sup_{0\leq s \leq t} \|j_0(s,\cdot)-j(s,\cdot)\|.
\]
In conjunction with (\ref{e:picest0}) this leads to the estimate 
\[
\sup_{0\leq s \leq t} \|j_{n}(s,\cdot)-j(s,\cdot)\|\leq \frac{t^n	}{\alpha^n n!} \|f\|,
\]
which yields the claim. 
\end{proof}

Although it is difficult to predict the tail behaviour of $(\cJ^{\alpha}_tf,f)$ as $t \rar \infty$ due to the oscillatory nature of the Bessel functions of the first kind, the Laplace transform of $(\cJ^{\alpha}_tf,g)$ is a familiar object. Thus the tail behaviour can be determined by inverting this Laplace transform as well. 
\begin{proposition} Suppose $f \in L^2(m)$. Then for all $s \geq 0$
\[
\int_0^{\infty} e^{-st} (\cJ^{\alpha}_tf,f)dt= \frac{1}{s}U^{\alpha+\frac{1}{s}}(f,f).
\]
\end{proposition}
\begin{proof}
Using the spectral representation of $\cJ^{\alpha}$
\bean
\int_0^{\infty} e^{-st} (\cJ^{\alpha}_tf,f)dt&=& \int_0^{\infty} e^{-st}\int_0^{\infty} \frac{1}{\lambda+\alpha}e^{-\frac{t}{\lambda+\alpha}}d(E_{\lambda}f,f)\\
&=&\int_0^{\infty} \frac{1}{\lambda+\alpha}\frac{1}{s+ \frac{t}{\lambda+\alpha}}d(E_{\lambda}f,f)\\
&=&\frac{1}{s}\int_0^{\infty} \frac{1}{\lambda+\alpha+\frac{1}{s}}d(E_{\lambda}f,f)=\frac{1}{s}U^{\alpha+\frac{1}{s}}(f,f).
\eean
Also observe that the above identity is valid for $s=0$ since $\alpha U^{\alpha}f \rar f$ as $ \alpha \rar \infty$ and
\[
\int_0^{\infty} (\cJ^{\alpha}_tf,f)dt=(f,f).
\]
\end{proof}
 When $X$ is a one-dimensional transient diffusion we have yet another way of characterising $\cJ^{\alpha}$.
\begin{proposition} Suppose that $X$ is as in Example \ref{ex:1dim} and is transient. Let $G_x:=\sup\{t \geq 0: X_t=x\}$ be the last hitting time of $x$. Then
\[
\cJ^{\alpha}f(x) = u(x,x)E^{\mu}\left(J_0(2 \sqrt{tG_x}) e^{-\alpha G_x}\chf_{[G_x>0]}\right),
\]
where $\mu$ is a measure on $(\bfE,\sE)$ given by $\mu(dy)=f(y)m(dy)$, and $u$ is the potential kernel for $X$, i.e.
\[
u(x,y)=\int_0^{\infty}p(t,x,y)dt.
\]
\end{proposition}
\begin{proof}
It is well-known that (see, e.g., p.27 of \cite{BorSal}) 
\[
P^y(0<G_y\leq t)=\int_0^t \frac{p(s,x,y)}{u(y,y)}ds.
\]
In view of the symmetry of $p(t,x,y)$ the above implies for all bounded and continuous $h$ that
\[
E^{\mu} h(G_x)\chf_{[G_x>0]}=\int_0^{\infty} h(s)\frac{P_sf(x)}{u(x,x)}ds,
\]
which yields the claim.
\end{proof}
Recall (see Chapter 9 of \cite{AS}) that $J_0$ satisfies the following ODE:
\be \label{e:J0ode}
x^2 J_0''+ x J_0' + x^2 J_0=0.
\ee
The above equation and its connection with $2$-dimensional Bessel process leads to the following remarkable observation that $\cJ^{\alpha}_t$ can be considered as the solution of a backward partial differential equation with an initial condition.
\begin{proposition}
Fix an $f \in L^2(m)$, $T>0$ and consider the following function
\be \label{e:h}
h(t,x)= \int_0^{\infty} J_0(2\sqrt{xs}) e^{-2(T-t)s}(P_sf,f)ds, \qquad x \geq 0,\,  t\in [0,T).
\ee
Then
\bea
h_t + 2x h_{xx} + 2 h_x &=&0; \label{e:pdeh}\\
h(0,\cdot)&=&(\cJ^T_x f,f).\nn
\eea
Moreover, $(h(t,X_t))_{t \in [0,S]}$ is a bounded martingale for any $S<T$ when $X$ is a $2$-dimensional squared Bessel process, i.e. $X$ is the unique weak solution
\[
dX_t =2\sqrt{X_t}dW_t +2 dt,
\]
where $W$ is a standard Brownian motion.
\end{proposition} 
\begin{proof}
First note that $|J_0|<1$ and $\frac{J_0'(x)}{x}$ is bounded on $[0,\infty)$. The latter implies that $\frac{d}{dx}J_0(2\sqrt{xs})$ is bounded whenever $s$ belongs to a bounded interval. In view of (\ref{e:J0ode}) these observations further yield that $\frac{d^2}{dx^2}J_0(2\sqrt{xs})$ is bounded when $(x,s)$ belong to compact squares. Thus, we can differentiate  under the integral sign in (\ref{e:h}) to get
\bean
&&h_t + 2x h_{xx} + 2 h_x\\
&=&\int_0^{\infty} \left\{2x \frac{d^2}{dx^2}J_0(2\sqrt{xs})+2 \frac{d}{dx}J_0(2\sqrt{xs}) + 2sJ_0(2\sqrt{xs})\right\}e^{-2(T-t)s}(P_sf,f)ds.
\eean
However, (\ref{e:J0ode}) implies that the term within the curly brackets vanishes. Moreover, $h(0,\cdot) = (\cJ^T_x f,f)$ by the definition of $h$. This completes the proof that $h$ solves the PDE in (\ref{e:pdeh}).

To finish the proof note that $(h(t,X_t))_{t \in [0,S]}$ is a local martingale by an application of Ito's formula. Moreover, for any $t\leq S$
\[
0\leq h(t,X_t)\leq \int_0^{\infty} e^{-2(T-S)s}(P_sf,f)ds= (U^{2(T-S)}f,f)<\infty,
\]
which in turn yields that $(h(t,X_t))_{t \in [0,S]}$ is a bounded martingale. 
\end{proof}
\section{Regularisation of the inverse problem}
Regularisation of inverse problems are in principle  perturbations of the forward operator so that its inverse becomes a bounded operator on the underlying Hilbert space.  As a bounded operator the perturbed  inverse operator can then be applied to any member of the Hilbert space. If the perturbation is small, one expects not to deviate too much from the solution of the original inverse problem, if exists. We refer the reader to \cite{ehn} for an exhaustive treatment of regularisation methods for inverse problems.

The most common method for regularising ill-posed inverse problems is the {\em Tikhonov regularisation}. In our set up this will correspond to the solution of an auxiliary problem
\[
P_t f +\gamma f = g, \qquad \gamma>0,\, g\in L^2(m).
\]
Using spectral calculus it can be formally showed that the inverse of $P_t +\gamma I$ is given by
\[
\int_0^{\infty}\frac{1}{\gamma +e^{-\lambda t}}dE_{\lambda}.
\]
Since $\gamma +e^{-\lambda t}$ is bounded away from $0$, this inverse operator is bounded and, therefore, has all of $L^2(m)$ as its domain. 

In view of the above heuristic discussion we shall next describe a family of perturbations of the original problem that results in a regularisation. The resulting problems can be viewed as a mixture of the original inverse problem with a suitable regularising noise. 
\begin{theorem} \label{t:regularisation}
Suppose that $\phi:\bbR_+ \mapsto \bbR_+$ is a continuous function with $\liminf_{x \rar \infty}\phi(x) >0$ such that $\sup_{x\geq 0} e^{-tx}\phi(x) <\infty$. Then, there exists a unique solution $f\in  L^2(m)$ to the following for any $g \in L^2(m)$ and $t>0$:
\be \label{e:reggen}
(1-\gamma)P_t f +\gamma \phi(-A)f=g, \qquad \gamma \in (0,1).
\ee
Moreover,  the solution is given by 
\be \label{e:regsol}
f=\int_0^{\infty}\frac{1}{\gamma \phi(\lambda)+(1-\gamma)e^{-\lambda t}}dE_{\lambda}g,
\ee
and has the property that 
\be \label{e:minimizer}
(1-\gamma)f= \argmin_{h\in L^2(m)}\|P_t h -g\|^2 +\frac{\gamma}{1-\gamma}(P_t \phi(-A)h,h).
\ee
\end{theorem} 
\begin{proof}
Observe that $f$ given by (\ref{e:regsol}) is well-defined and belongs to $L^2(m)$ since $g\in L^2(m)$ and $\liminf_{x \rar \infty}\phi(x) >0$. Moreover, it belongs to the domain of $\phi(-A)$. The fact that  $f$ is the solution of (\ref{e:reggen}) is easy. Indeed, using the spectral representation 
\[
(1-\gamma)P_t f +\gamma \phi(-A)f=\int_0^{\infty}\frac{e^{-\lambda t}(1-\gamma)+ \gamma\phi(\lambda)}{\gamma \phi(\lambda)+(1-\gamma)e^{-\lambda t}}dE_{\lambda}g=g.
\]
Thus, it remains to show (\ref{e:minimizer}). 

\noindent {\em Case 1:} First suppose that $A$ is a finite rank operator, which is equivalent to saying that $P_t$ is a finite rank operator. Assume that the dimension of the range of $A$ is $n\in \mathbb{N}$. This implies  the existence of an orthonormal family $(\nu_k)_{k=1}^n \subset L^2(m)$ and real numbers $(\lambda_k)_{k=1}^n$ such that for any $h \in L^2(m)$
\[
-Ah =\sum_{k=1}^n \lambda_k (\nu_k,h)\nu_k.
\]
The corresponding spectral family is defined via
\bean
E_{\lambda}h&=&\sum_{\lambda_k\leq \lambda}(\nu_k,h)\nu_k,\; \lambda>0;\\
h_0:=E_{0}h&=&h-\sum_{k=1}^n (\nu_k,h)\nu_k.
\eean
Consequently,
\bean
&&\|P_t h -g\|^2 +\frac{\gamma}{1-\gamma}(P_t \phi(-A)h,h)\\
&&=\sum_{k=1}^n \left(e^{-2 \lambda_k t} (h, \nu_k)^2 -2 e^{-\lambda_k t} (h,e_k)(g,\nu_k) + (g,\nu_k)^2 +\frac{\gamma}{1-\gamma}e^{-\lambda_k t}\phi(\lambda_k)(h,\nu_k)^2\right)\\
&&+\int_{\bfE}\left((h_0-g_0)^2+\frac{\gamma}{1-\gamma}\phi(0)h_0^2\right)dm.
\eean
Minimising the quadratic in every summand we deduce that the minimiser, $\hat{h}$, satisfies
\[
(\hat{h},e_k)=\frac{e^{-\lambda_k t} (g,e_k)}{e^{-2\lambda_k t} +\frac{\gamma}{1-\gamma}e^{-\lambda_k t} \phi(\lambda_k)}=\frac{(g,e_k)}{e^{-\lambda_k t} +\frac{\gamma}{1-\gamma} \phi(\lambda_k)}
\]
Moreover, minimising the integrand that is quadratic in $h_0$ we see for $m$-a.e. $y\in \bfE$
\[
\hat{h}_0(y)=\frac{g_0(y)}{1+\frac{\gamma}{1-\gamma} \phi(0)}.
\]
Thus,
\[
(1-\gamma)\hat{h}=\int_0^{\infty}\frac{1}{\gamma \phi(\lambda)+(1-\gamma)e^{-\lambda t}}dE_{\lambda}g,
\]
which establishes (\ref{e:minimizer}). 

Note that (\ref{e:regsol}) can be written alternatively that
\[
f= F(-A),
\]
where $F:\bbR_+ \mapsto \bbR_+$ is the following bounded continuous function:
\[
F(x)=\frac{1}{\gamma\phi(x)+(1-\gamma)e^{-xt}}.
\]
Thus, we have shown
\be \label{e:altrep}
(1-\gamma)F(-A)g=\argmin_{h\in L^2(m)}\|P_t h -g\|^2 +\frac{\gamma}{1-\gamma}(P_t \phi(-A)h,h)
\ee
when $A$ is a finite rank operator.
This alternative representation is going to be useful in extending the validity of (\ref{e:minimizer}) for a general $A$. 

\noindent {\em Case 2:} Suppose that $A$ is a bounded operator. Define a sequence of finite rank self-adjoint nonpositive operators, $(A_n)$, via 
\[
A^n e_k =A e_k, \, k\leq n, \quad A^n e_k =0, \, k\geq n+1,
\]
where $(e_k)_{k=1}^{\infty}$ is an orthonormal basis for the Hilbert space $L^2(m)$.  Then, $A^n\rar A$ strongly, i.e. for any $h\in L^2(m)$, $A^nh \rar A^h$ in $L^2(m)$. Denote the associated semigroups by $(P^n_t)$ and define $S^n:L^2(m) \mapsto \bbR_+$ and $S:L^2(m) \mapsto \bbR_+$ as follows:
\bean
S_n(h)&=&\|P^n_t h -g\|^2 +\frac{\gamma}{1-\gamma}(P^n_t \phi(-A^n)h,h)\\
S(h)&=&\|P_t h -g\|^2 +\frac{\gamma}{1-\gamma}(P_t \phi(-A)h,h).
\eean
Since $P^n_t=\exp(tA^n), P_t=\exp(tA)$, $\phi$ is continuous, and $\sup\|A_n\|<\infty$, it follows that $P^n_t$ and $P^n_t \phi(-A^n)$ converge strongly to $P_t$ and $P_t \phi(-A)$, respectively (see Corollary 2 on p.2220 of \cite{DS3}). Therefore, we have
\be \label{e:lboundmin}
S(h)=\lim_{n \rar \infty} S_n(h) \geq \lim_{n \rar \infty} S_n((1-\gamma)F(-A^n)g),
\ee
where the inequality is due to (\ref{e:altrep}) since $A^n$ is a finite rank operator. 

Let $h_n:=F(-A^n)g$ and observe that $h_n \rar F(-A)g$ in $L^2(m)$ since $F$ is continuous. Thus,  $(P^n_t \phi(-A^n)h_n,F(-A)g) \rar (P_t \phi(-A)h_n,F(-A)g)$ as $n \rar \infty$ as $(P_t^n)$ and $(A^n)$ are strongly convergent to $P_t$ and $A$. On the other hand,
\[
(P^n_t \phi(-A^n)h_n,h_n -F(-A)g)\leq K \|h_n-F(-A)g\|^2
\]
for some $K<\infty$ $P^n_t \phi(-A^n)$ are uniformly bounded operators due to the assumption that $e^{-tx}\phi(x)$ is bounded.  Since $\|h_n-F(-A)g\|^2 \rar 0$ as $n \rar \infty$,  we may now conclude that $(P^n_t \phi(-A^n)h_n,h_n) \rar (P_t \phi(-A)h_n,F(-A)g)$ as $n \rar \infty$. By similar arguments we can also show that $\|P^n_th_n-g\|\rar \|P_t F(-A)g-g\|$. Therefore, $S_n((1-\gamma)F(-A^n)g) \rar S((1-\gamma)F(-A)g$, which in turn implies
\[
S(h)\geq S((1-\gamma)F(-A)g), \quad \forall h \in L^2(m)
\]
in view of (\ref{e:lboundmin}). Since $F(-A)g \in L^2(m)$, we have
\[
(1-\gamma)F(-A)g=\argmin_{h\in L^2(m)}\|P_t h -g\|^2 +\frac{\gamma}{1-\gamma}(P_t \phi(-A)h,h)
\]
\noindent {\em Case 3:} Given a general $A$ define
\[
A^n:=\int_0^n \lambda dE_{\lambda}, \qquad n \geq 1.
\]
Note that each $A^n$ is a bounded operator. Thus, (\ref{e:minimizer}) holds when $A$ is replaced by $A^n$ and $P_t$ by $P^n_t$, where $(P^n_t)$ is the associated semigroup generated by $A^n$.   Moreover, $(P^n_t \phi(-A^n))$ is a sequence of uniformly bounded operators converging strongly to $(P_t \phi(-A)$ as $n \rar \infty$. Thus, imitating the proof of the previous case we can similarly establish that (\ref{e:minimizer}) is satisfied.
\end{proof}
\begin{remark} The assumption that $\liminf_{x \rar \infty}\phi(x) >0$ cannot be dispensed easily if (\ref{e:reggen}) is to have a solution for any given $g \in L^2(m)$. To wit take $\phi(x)=e^{-tx}$. Then  (\ref{e:reggen}) becomes $P_tf=g$, which does not have a solution in general. 
\end{remark}
Observe that $\phi(x)=e^{tx}$ satisfies the conditions in Theorem \ref{t:regularisation} and $P_t \phi(-A)$ is the identity operator. This observation leads to the following corollary. 
\begin{corollary} \label{c:regularisation}
There exists a unique solution $f\in  L^2(m)$ to the following for any $g \in L^2(m)$ and $t>0$:
\be \label{e:reggenc}
(1-\gamma)P_t f +\gamma P_t^{-1}f=g, \qquad \gamma \in (0,1).
\ee
Moreover,  the solution is given by 
\be \label{e:regsolc}
f=\int_0^{\infty}\frac{1}{\gamma e^{t\lambda}+(1-\gamma)e^{-\lambda t}}dE_{\lambda}g,
\ee
belongs to $\cD(P_t^{-1})$, and has the property that 
\be \label{e:minimizerc}
(1-\gamma)f= \argmin_{h\in L^2(m)}\|P_t h -g\|^2 +\frac{\gamma}{1-\gamma}(h,h).
\ee
\end{corollary} 
 If $g\in \cD(P_t^{-1})$, one should expect that the solutions of (\ref{e:reggen}) converge to $P_t^{-1}g$ as $\gamma \rar 0$. This is indeed the case as the following proposition shows. 
\begin{proposition} Let $\phi$ be as in Theorem \ref{t:regularisation} and for each $\gamma \in (0,1)$ denote by $f_{\gamma}$ the solution of (\ref{e:reggen}). Assume further that $g \in \cD(P_t^{-1})$. Then
\[
\lim_{\gamma \rar 0} \|f_{\gamma}-P_t^{-1}g\|=0.
\]
\end{proposition}
\begin{proof}
The hypothesis that $g \in P_t^{-1}$ implies
\[
\int_0^{\infty} e^{2 \lambda t}d(E_{\lambda}g,g)<\infty.
\]
On the other hand,
\[
\left(e^{\lambda t} -\frac{1}{\gamma \phi(\lambda)+(1-\gamma)e^{-\lambda t}}\right)^2 \leq \left(\frac{2-\gamma}{1-\gamma}\right)^2 e^{2 \lambda t}.
\]
Thus, in view of the Dominated Convergence Theorem, we have
\[
\lim_{\gamma\rar 0}\int_0^{\infty}\left(e^{\lambda t} -\frac{1}{\gamma \phi(\lambda)+(1-\gamma)e^{-\lambda t}}\right)^2d(E_{\lambda}g,g)=0,
\]
which yields the claim.
\end{proof}
Although looking abstract, Theorem \ref{t:regularisation} furnishes us with a plethora of concrete examples for regularising the inverse problem (\ref{e:IP}). To see this in a specific example suppose that the transition function, $(P_t)$, possesses a density with respect to $m$. Let us denote this transition density with $p(t,\cdot,\cdot)$ and introduce a new operator, $B$, on $L^2(m)$ via
\be \label{e:mixgen}
Bf(x):= \int_{\bfE} f(y)p(T^*,x,y)m(dy)-f(x)=P_{T^*}f(x) -f(x),
\ee
where $T^*>0$ is fixed. Due to the symmetry of $P_{T^*}$, $B$ is a also bounded symmetric operator on $L^2(m)$. Moreover, it corresponds to the generator of a Markov jump process that remains constant between the jumps of a Poisson process with unit parameter and moves between the points of $\bfE$ according to the transition function $P_{T^*}$ (see Example \ref{ex:jump}). Thus, by enlarging the probability space if necessary, we can assume the existence of a Markov jump process, $J$,  with generator $B$ and independent from $X$. The semigroup, $(\tilde{P}_t)$, associated with $J$  is easily seen to satisfy $\tilde{P}_t=e^{tB}=\phi(-A)$, where
$\phi(x)=\exp(t(e^{-tx}-1))$. Clearly, $\phi$ satisfies the conditions of Theorem \ref{t:regularisation}. Thus, $L^2(m)=\cD(Q_t^{-1})$, where 
\[
Q_t= (1-\gamma)P_t + \gamma \tilde{P}_t.
\]
Note that  $(Q_t)$ is the semigroup of the Markov process, $Y$, where
\[
Y= \xi X+ (1-\xi)J,
\]
and $\xi$ is a Bernoulli random variable independent of $X$ and $J$ with $Prob(\xi=1)=1-\gamma$. Therefore, mixing the original Markov process with a pure jump process we observe that the inverse problem admits a solution. This construction readily extends to the following result. 
\begin{corollary} \label{c:mixing} Suppose that $K$ is a bounded positive operator such that $K=\psi(-A)$ for some bounded continuous function $\psi:\bbR_+\mapsto [0,1]$. In an enlargement of the probability space there exists a Markov process $Y$ such that
\[
Y= \xi X+ (1-\xi)J,
\]
where $\xi$ is a non-degenerate Bernoulli random variable, $J$ is a jump Markov process with generator
\[
Bf=Kf-f, \qquad f \in L^2(m),
\]
and $\xi, J$ and $X$ are mutually independent.
Moreover, $\cD(Q_t^{-1})=L^2(m)$, where $(Q_t)$ is the semigroup associated to $Y$.
\end{corollary}
In view of the relationship between the inverse problem and the backward PDEs the above corollary leads to the following in view of Corollary \ref{c:invpde}.
\begin{corollary} \label{c:mixingPDE} Suppose that $K$ is a bounded positive operator such that $K=\psi(-A)$ for some bounded continuous function $\psi:\bbR_+\mapsto [0,1]$. Then, for any $g \in L^2(m)$, there exists a unique solution on  
\[
u_t+ \xi Au+ (1-\xi)(Ku-u)=0, \qquad u(0,\cdot)=g.
\]
\end{corollary}
The above corollaries show that if we construct a new process by randomly mixing the original process with a suitably chosen independent  jump process, the inverse problem becomes well-posed when $P_t$ is replaced with the corresponding transition function of the new process. Note that $\gamma$ can be chosen arbitrarily close to 1, which in practice means that one would almost never see the jump process, $Y$.

 In particular when the generator is a differential operator Corolllary \ref{c:mixingPDE} shows the existence and uniqueness of a solution for the following partial integro-differential equation for any $\xi \in (0,1) $, $T^*>0$, and $g \in L^2(m)$, which can be viewed as the regularisation for the ill-posed PDE in (\ref{e:bpde}).
 \bea \label{e:r-pide}
 u_t(t,x) + \xi Au(t,x) +(1-\xi) \int_{\bfE} u(t,y) P(T^*,x,dy)-(1-\xi)u(t,x)&=&0;\\
 u(0,\cdot)&=&g. \nn
 \eea
\begin{example} Suppose that $X$ is a Brownian motion and $m$ is the Lebesgue measure on the real line. Then, the generator $B$ defined in (\ref{e:mixgen}) corresponds to a compound Poisson process with unit intensity, whose jumps are normally distributed with mean $0$ and variance $T^*$. In this case the process $Y$ of Corollary \ref{c:mixing} is a Brownian motion with probability $1-\gamma$ and a compound Poisson process with probability $\gamma$. 

Corollary \ref{c:mixingPDE}, on the other hand, gives us a regularisation of the ill-posed backward heat equation with an initial condition. The regularisation takes the form of a partial integro-differential equation as follows:
\bean
u_t +\frac{\xi}{2}u_{xx} + (1-\xi)\int_{-\infty}^{\infty}(u(t,y)-u(t,x))\frac{1}{\sqrt{2 \pi T^*}}\exp\left(-\frac{(x-y)^2}{2 T^*}\right)dy&=&0;\\
 u(0,\cdot)&=&g.
\eean
Corollary \ref{c:mixingPDE} yields the existence and uniqueness of a solution to the above for {\em any} $g \in L^2(m)$.
\end{example}
\begin{example} Let $K=\alpha U^{\alpha}$ for some $\alpha>0$, where $U^{\alpha}$ is the $\alpha$-potential operator. Observe that $\|K\|\leq 1$ so Corollary \ref{c:mixing} is applicable.  In this case the process $Y$ is given by $X$ with probability $1-\gamma$ while it is equal to a Markov jump process with generator $K-I$ with probability $\gamma$.
\end{example} 


\begin{thebibliography}{99}
\bibliographystyle{plain}
\bibitem{AS} Abramowitz, M. and Stegun,I.A. (1972): {\em  Handbook of mathematical functions: with formulas, graphs, and mathematical tables}. National Bureau of Standards Applied Mathematics Series. Vol. 55. Corrected 10th printing. 
\bibitem{imaging} Bertero, M. and Boccacci, P. (1998): {\em Introduction to inverse problems in imaging}. CRC press.
\bibitem{BorSal} Borodin, A. N., and Salminen, P. (2012): {\em Handbook of Brownian motion-facts and formulae}, Birkh\"auser.
\bibitem{Cavalier} Cavalier, L. (2011): Inverse problems in statistics. {\em Inverse problems and high-dimensional estimation}, Lect. Notes Stat. Proc., 203, pp. {3--96}.
\bibitem{DS3} Dunford, N. and  Schwartz, J.T. (1971): {\em  Linear Operators. Part 3: Spectral Operators}, Wiley-Interscience.
\bibitem{eggL} Eggermont, P. P. B. and LaRiccia, V. N. (1995): Maximum smoothed likelihood density estimation for inverse problems.  {\em Ann. Statist.}, 23(1), pp. 199--220.
\bibitem{ehn} Engl, H.W., Hanke, M. and Neubauer, A.. (1996). {\em Regularization of inverse problems},  Vol. 375,  Springer Science \& Business Media.
\bibitem{Ekstrom} Ekstr{\"o}m, E., Hobson, D., Janson, S. and Tysk, J. (2013): Can time-homogeneous diffusions produce any distribution? {\em Probab. Theory Related Fields}, 155(3-4), pp. {493--520}.
\bibitem{jelliott} Elliott, J. (1955): Eigenfunction expansions associated with singular differential operators. {\em Trans. Amer. Math. Soc.}, 78, pp.  {406--425}.
\bibitem{EK} Ethier, S., and Kurz, T. (1989): {\em Convergence of Markov Processes}, Wiley.
\bibitem{falkner} Falkner, N. (1983): Stopped distributions for Markov processes in duality. {\em Zeitschrift f{\"u}r Wahrscheinlichkeitstheorie und Verwandte Gebiete}, 62(1), pp. 43--51.
\bibitem{FOT} Fukushima, M., Oshima, Y.  and  Takeda, M. (2010): {\em Dirichlet forms and symmetric Markov processes}. Studies in Mathematics, Vol. 19,  de Gruyter.
\bibitem{IK} Ito, K., McKean, H. P., Jr. (1965): {\em Diffusion processes and their sample paths}, Springer. 
\bibitem{KooChung} Koo, J-Y. and Chung, H-Y. (1998): Log-density estimation in linear inverse problems. {\em Ann. Statist.}, 26(1), pp. 335--362. 
\bibitem{MK} McKean, Jr., H. P. (1956): Elementary solutions for certain parabolic partial differential equations. {\em Trans. Amer. Math. Soc.}, 82, pp. {519--548}.
\bibitem{Silvermanetal} Silverman, B. W., Jones, M. C., Nychka, D. W. and Wilson, J. D.. (1990): A smoothed {EM} approach to indirect estimation problems, with particular reference to stereology and emission tomography. {\em J. Roy. Statist. Soc. Ser. B}, 52(2), pp. {271--324}.
\bibitem{VL} Vardi, Y., and Lee, D. (1993) From image deblurring to optimal investments: Maximum likelihood solutions for positive linear inverse problems. {\em Journal of the Royal Statistical Society. Series B (Methodological),} pp. 569--612.
\bibitem{WidderTams} Widder, D. V. (1951): Necessary and sufficient conditions for the representation of a function by a {W}eierstrass transform. {\em Trans. Amer. Math. Soc.}, 71, pp. {430--439}.
\end{thebibliography}
\end{document}